\documentclass[12pt]{article} 

\usepackage[utf8]{inputenc} 

\usepackage{geometry} 
\usepackage{graphicx} 

\usepackage{booktabs} 
\usepackage{array} 
\usepackage{paralist} 
\usepackage{verbatim} 
\usepackage{subfig} 

\usepackage{fancyhdr} 
\usepackage{sectsty}
\allsectionsfont{\sffamily\mdseries\upshape} 

\usepackage[nottoc,notlof,notlot]{tocbibind} 
\usepackage[titles,subfigure]{tocloft} 

\usepackage{url} 
\usepackage{amsthm, amssymb,amsmath} 
\newtheorem{theorem}{Theorem}
\newtheorem{prop}{Proposition}

\newtheorem{corollary}{Corollary}
\newtheorem*{conjecture}{Conjecture}

\title{The Lower Bound for Number of Hexagons in Strongly Regular Graphs with Parameters $\lambda=1$ and $\mu=2$}
\author{Reimbay Reimbayev}
\date{} 

\begin{document}
\maketitle

\begin{abstract}
The existence of $srg(99,14,1,2)$ has been a question of interest for several decades to the moment. In this paper we consider the structural properties in general for the family of strongly regular graphs with parameters $\lambda =1$ and $\mu =2$. In particular, we establish the lower bound for the number of hexagons and, by doing that, we show the connection between the existence of the aforementioned graph and the number of its hexagons.
\end{abstract}

\section{Introduction}

The existence of some graphs, most notably of those that have a very fine structure called strong regularity, is still not known  \cite{Gordon, Brouwer}. In his renowned set of five problems, John Conway \cite{Conway} stated a problem regarding the search for one of such graphs. It states the following.

\textbf{Problem:} Is there 99-vertex graph such that the following conditions are satisfied:
I. Any edge belongs to a unique triangle ($C_3$);
II. Any non-edge belongs to a unique quadrilateral ($C_4$)?

The problem is a rephrase of the search for a strongly-regular graph with parameters $n=99, k=14, \lambda = 1, \mu = 2, $ in short, $srg(99,14,1,2)$. Makhnev \cite{Makhnev} has answered this question partially. In this paper, we study the structure of such graphs in general and find the lower bound for number of hexagons; in doing so, we show that if the lower bound for hexagons achieved then the graph doesn't exist. Supporting this hypothesis is the fact that both of the known graphs of the same class, namely $srg(9,4,1,2)$ (i.e. Paley 9) and $srg(243, 22, 1,2)$, obey to it and take the lowest possible value for the number of hexagons.

Moreover, it does look like if the other two graphs in the same class with parameters $k=112$ and $k=994$ should exist have to be built of Paley 9, i.e. $srg(9,4,1,2)$, as building blocks. But without strict proves we can only speculate about it.

\section{Preliminary Study}

For simplicity, a graph, satisfying conditions I and II without regard to its order $n$,  henceforth be denoted $G$. First of all, let us show that the graph $G$ is indeed an srg. Obviously, $G$ is simple due to Condition I and with at least two vertices is connected due to Condition II. Also conditions guaranty that if $G$ is regular than it is strongly regular. Thus, we just need to prove the regularity.

\begin{prop}
Graph $G$  is regular, thus - strongly regular.
\end{prop}

\begin{proof}
We can safely assume that $G$ has at least two vertices or else, it does not have any edges or non-edges. Choose vertex $a \in V(G)$, $G$ is connected so there exists $b \in G$ s.t. $ab \in E(G)$. Condition I guarantees the existence of the unique $c$ - the third vertex of the triangle with vertices $a,b,c$. Denote $N(v)$ - the set of vertices adjacent to a given vertex $v$, its neighborhood excluding the vertex itself. If $N(a) \setminus \{a,c\}=N(b) \setminus \{ b,c\}= \varnothing$ then $G=K_3$ thus 2-regular and by default is an $srg(3,2,1,2)$.
Otherwise choose $v \in N(a)\setminus \{ b,c \}$. $vb \notin E(G)$ and $vc \notin E(G)$ due to Condition I. As $vb$ is a non-edge, Condition II identifies $w \in N(b)$ s.t. $vw \in E(G)$. $v$ is not adjacent to any other vertices from $N(b) \setminus \{ a,c\}$. Similarly $w$ cannot be adjacent to any other vertices from $N(a) \setminus \{ b,c\}$. Bijection has been established, which means the vertex degrees are equal, $d_a=d_b$, i.e. $a$ and $b$ are vertices of equal valencies. $G$ is connected, thus for any $x \in V(G)$ there is a path $a,y,...,x$, with $d_a=d_y=...=d_x$, so $d_a=d_x$. $G$ is regular.
\end{proof}

\begin{figure}
	\includegraphics[width=0.6\textwidth]{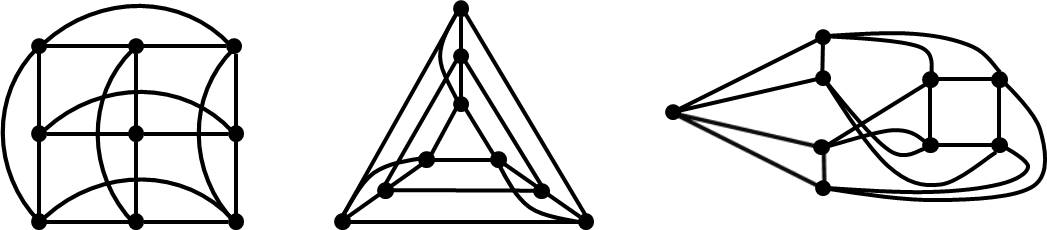}
		\centering
		\caption{Drawings of Paley 9 graph $P_9\equiv srg(9,4,1,2)$.}
		\label{fig1}
\end{figure}

As a result of the proposition, $G$ is strongly regular, $srg(n,k,\lambda,\mu)$, where $n$ -number of vertices, $k$- valency, that is necessarily even due to Condition I, $\lambda = 1$ (also Condition I), $\mu = 2$ (Condition II). Thus Conditions I and II define the class of strongly regular graphs with parameters $\lambda = 1$ and $\mu = 2$. The order of the graph $n$ and its valency $k$ are also related with a simple formula.

\begin{prop}
For a $k$-regular graph $G$ satisfying Conditions I and II, \[n=|V(G)|=\frac{k^2+2}{2}\].
\label{prop2}
\end{prop}

\begin{proof}
We will use the standard technique in Graph Theory called Double Counting. Let $a \in V(G)$, $N(a)$- neighborhood of $a$, and $W(a)=V(G)\setminus N(a) \setminus \{a\}$, the set of vertices of $G$ different from $a$ and $N(a)$.
So we have: $n=|V(G)|$, $|N(a)|=k$, $|W(a)|=n-k-1$. Consider all the edges between $N(a)$ and $W(a)$. Then,
\[k(k-2)=2(n-k-1).\]
Left-hand side is due to regularity and Condition I; the right-hand side- due to Condition II. Solving the equation, we get $n=\frac{k^2+2}{2}$.
\end{proof}

Denote $p_3, p_4, p_5$ and $p_6$ the number of, respectively, triangles (induced subgraphs isomorphic to cycle $C_3$), quadrilaterals ($C_4$), pentagons ($C_5$) and hexagons ($C_6$). The next few proposition are about the number of such polygons (cycles) in $G$. The quantities $p_3$, and $p_4$ can be found directly.

\begin{prop}
Graph $G$ has exactly $\frac{1}{6}nk$ triangles and $\frac{1}{8}nk(k-2)$ quadrilaterals.
\end{prop}

\begin{proof}
Straight-forward counting using Condition I and Handshaking Lemma gives:
\[p_3 =\frac{\frac{nk}{2}}{3}=\frac{nk}{6}.\]
To count the number of quadrilateral we will use the fact that each node of $G$ has $n-k-1$ nodes (vertices) non-adjacent with it. Condition II guarantees exactly one quadrilateral for each of them. Counting over all the vertices and dividing to four, because we count each quadrilateral exactly four times, we obtain: 
\[p_4 = \frac{n(n-k-1)}{4}=\frac{1}{4}n(\frac{k^2+2}{2}-k-1)=\frac{1}{8}nk(k-2).\]
Notice that we have used Proposition \ref{prop2} here.
\end{proof}

To find $p_5$ we need to work a bit harder.

\begin{theorem}
Graph G has exactly $\frac{1}{5}nk(k-2)(k-4)$ pentagons.
\end{theorem}

\begin{proof}
Any closed walk of length 5 in the graph can be coded into a string of six numbers $d_1 d_2 d_3 d_4 d_5 d_1$, which we denote here the distance between a vertex in a walk to its starting, and thus finishing, vertex. Pentagons will be coded by the string $0 1 2 2 1 0$. Using the geometry of the graph, the first vertex can be chosen $n$ ways, the second - $k$ ways, the next three respectively $k-2, k-2, 2$ ways and the last vertex, being already predetermined, - one way. In total, there are exactly $n\cdot k\cdot (k-2)\cdot (k-2)\cdot 2 \cdot 1=2nk(k-2)^2$ of such walks. Except of pentagons, two more possible configurations, $T_1$ and $T_2$, satisfy the same code (Figure \ref{fig2}).

\begin{figure}
	\includegraphics[width=1.0\textwidth]{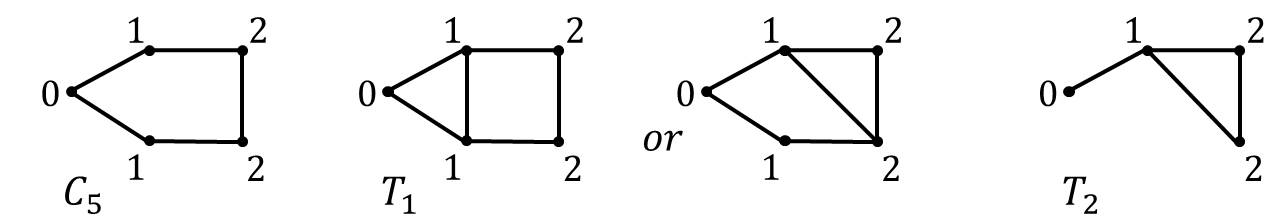}
		\centering
		\caption{Walks coded by string $0 1 2 2 1 0$ and their induced subgraphs.}
		\label{fig2}
\end{figure}

Thus, denoting $t_1$ the number of subgraphs of type $T_1$, and $t_2$ - of type $T_2$, we have
\[
	2nk(k-2)^2=10\cdot p_5+6 t_1+2 t_2,
\]
where $t_1=4 \cdot p_4 $; and $t_2=3(k-2) \cdot p_3$. The coefficients in front of $t_1$ and $t_2$ are coming from the symmetries of the walks. 
 
So,
\begin{align*}
   10\cdot p_5 &=2nk(k-2)^2-6t_1-2t_2\\
                                          &=2nk(k-2)^2-6\cdot 4\cdot \frac{1}{4}n(n-k-1)-2\cdot 3(k-2)\frac{nk}{6}\\
                                          &=2nk(k-2)^2-6n(\frac{k^2+2}{2}-k-1)-nk(k-2)\\
                                          &=2nk(k-2)^2-3nk(k-2)-nk(k-2)\\
                                          &=2nk(k-2)(k-4).
\end{align*}

Notice that we have used Proposition \ref{prop2} in calculations. The statement follows.

\end{proof}

The next statement follows immediately.

\begin{corollary}
An edge of $G$ belongs to exactly $2(k-2)(k-4)$ pentagons.
\label{corollary1}
\end{corollary}
\begin{proof}
On average, each edge belongs to $p_5/|E(G)|=2(k-2)(k-4)$ pentagons, where $|E(G)|$ - number of edges in $G$. So we just need to prove that this is the maximum number of pentagons possible for a given edge.

Given $ab\in E(G)$, each of $k-2$ vertices of $a$, out of triangle based on $ab$, is adjacent to exactly one vertex from neighborhood of $b$ and is not adjacent to exactly $k-3$. Remember Condition II. Now it gives us at most $2(k-2)(k-3)-2(k-2)=2(k-2)(k-4)$ pentagons, where subtraction needed due to the two existing routs to each choice of $k-2$ vertices. The statement follows.
\end{proof}

\section{Main result}

To find the number of hexagons we will compare two quantities: the coefficient $c_6$ of the characteristic polynomial of the adjacency matrix of $G$ against the number of all possible triples of edges in $G$, which is obviously $\binom{|E(G)|}{3}$.

To begin with, we have to remind ourselves some known fact from algebraic graph theory. Given a graph $G$, with its adjacency matrix $A=A(G)$ and characteristic polynomial $P_G(x)$, the coefficients of its characteristic polynomial are connected with the structure of the graph in the following manner:

\begin{equation}
	c_i=(-1)^i\sum_{|S|=i} det A(G[S]),
	\label{eq1}
\end{equation}
\noindent where
\[
	P_G(x)=det(\lambda I - A)=\sum_{i=1}^n (x-\lambda_i)=c_0x^n+c_1x^{n-1}+c_2x^{n-2}... ,
\]
\noindent and $A(G[S])$ is an adjacency matrix of an induced subgraph on the set of vertices $S$ (West, 2-nd ed., p.454 \cite{West}). Here $\lambda_i$-s are the eigenvalues of $A[G]$.

Instead of vertices, we can induce the subgraphs on the set of three edges, which might not always give a subgraph of order six. The next proposition asserts those cases.

\begin{prop}
Denote $e_4$ - number of edge triples that are based on at most four vertices of $G$, $e_5$ - number of edge triples that are based on exactly five vertices of $G$.
The following equalities hold:
\begin{align*}
e_4&=\frac{1}{6}nk(4k^2-9k+3);\\
e_5&=\frac{1}{8}nk(k-2)(k^3+k^2-8k+2).
\end{align*}
\end{prop}

\begin{proof}
Three edges can be contained by three vertices if they are mutually incident and form a triangle in $p_3$ ways. Three edges can all be incident to exactly one vertex, and thus being contained by four vertices, - in $n\binom{k}{3}$ ways. And finally, three edges can be incident consequentially as in a path $P_4$: choose the middle edge arbitrarily from all possible edges; two adjacent ones - such that they do not form a triangle, - altogether, in $\frac{nk}{2}((k-1)^2-1)$ ways.
Collecting,
\[e_4=\frac{nk}{6}+n\binom{k}{3}+\frac{nk}{2}((k-1)^2-1)=\frac{1}{6}nk(4k^2-9k+3).
\]
Notice that different triples of edges can give the same induced subgraph, but it should not bother us at the moment as we are counting only distinct triples of edges, not subgraphs.

To find $e_5$, we have to realize that three edges can be incident to exactly five vertices only if two edges are incident while the third edge is not incident to the previous two. This fact means that among five vertices we always have one unique vertex with two edges incident to it. Choose that vertex, $n$ ways; next choose two incident to it edges out of $k$ possible. Here we have to consider two possibilities: when the pair of edges belong to a triangle, $\frac{k}{2}$ pairs, and when they do not, $\binom{k}{2}-\frac{k}{2}$ cases. For each possibility, we will choose the third edge out of all possible edges not incident to the ones already chosen.

Thus,
\begin{align*}
e_5&=n[ \frac{k}{2}(\frac{nk}{2}-3(k-2)-3)+(\binom{k}{2}-\frac{k}{2})(\frac{nk}{2}-(k-2)-2(k-1)-2)]\\
&=\frac{1}{2}nk(k-1)(\frac{nk}{2}-3k+2)+\frac{nk}{2}=\frac{1}{8}nk(k-2)(k^3+k^2-8k+2).
\end{align*}

\end{proof}

Now we turn our attention to the case when three edges incident to six vertices. It is the case when we have a perfect matching, or three-edge covers of the six vertex subgraphs. We have to consider all possible subgraphs on six vertices of $G$. They are given in the following two tables. Table 1 considers connected subgraphs and Table 2 - disconnected ones.

\begin{center}
\begin{tabular}{ |r|l|l|l|l|l|l|l|l|l|l|l|l|l|l|l|l|l|l|l|l|l|l|l|l|l|l|l|l|l|l|l|l|l|l|l|l|l|l|l|l| }
  \hline
  \multicolumn{12}{|l|}{Table 1: Connected subgraphs} \\
  \hline
 num.&1.   &2.    &3.    &4.    &       &       &5.   &       &6.   &       &7.   \\
 Cvet.  &51   &68   &70   &72   &79   &83   &84   &85   &86  &87   &88   \\
 det.    &0    &-4    &0     &-1    &-1   &-1    &3     & -1   &0   &-1    & 0   \\
 cov.    &4    &2     &2     &3     &1     &1     &1     &1     &2   &1     & 2    \\
  \hline
  \multicolumn{12}{|l|}{     } \\
  \hline
 num.&8.   &9.    &10.&     &       &       &       &       &      &       &       \\
 Cvet.  &89  &92   &93 &94  &95   &96   &97   &98   &99   &100 &101\\
 det.    &-4   &-1   &0   &0    &0    &-1     &-1    &-1   &0    &-1     & 0\\
 cov.   &2     &3    &2   &0    &0     &1     &1     &1     &0     &1     &0\\
  \hline
  \multicolumn{12}{|l|}{     } \\
  \hline
 num.&      &       &11.   &     & 12.  &     &     &       &       &       &       \\
 Cvet.  &102 &103 &104 &105 &106 &107&108&109 &110 &111  &112\\
 det.    &-1   &0     & 0    &-1   &-4    &0    &0    &0    &-1    &0     &-1\\
 cov.   &1     &0     & 2    &1     &2    &0    &0    &0     &1    &0      &1\\
  \hline   
\end{tabular}
\end{center}

Here, \emph{num.} - is the special numeration of the graphs that do not vanish when we add the bottom rows: $det.+cov.$; \emph{Cvet.} - the numeration of six-vertex graphs due to  Cvetcovic \cite{Cvetcovic}; \emph{det.} - is the determinant of the adjacency matrix of the given graph; \emph{cov.} - number of edge covers of the graph by exactly three edges. Here we have to use \cite{Cvetcovic} in order to make sure that we do not miss any graph. The paper gives a complete list of connected six-vertex graphs and allows us to refer to the graphs without necessarily drawing them. The values for the determinants have been found using online Matrix Calculator \cite{matrixCalc}. We will use them in summation for $c_6$ from formula (\ref{eq1}). Five more disconnected graphs on six vertices, that are not available from the list of Cvetcovic \cite{Cvetcovic}, are given in Table 2. 

\begin{center}
\begin{tabular}{ |r|l|l|l|}
  \hline
  \multicolumn{4}{|l|}{Table 2: Disconnected subgraphs} \\
  \hline
 num.&       graph                                                & det.    &cov. \\ \hline
          &\includegraphics[width=0.8in]{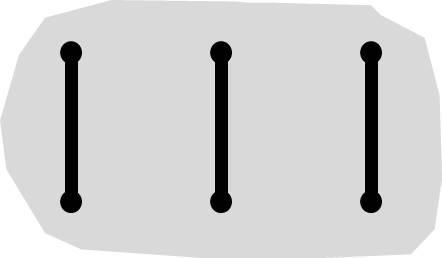}  &-1       &1  \\
          &\includegraphics[width=0.8in]{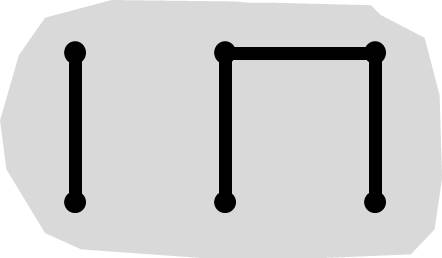}  &-1       &1    \\
 13.    &\includegraphics[width=0.8in]{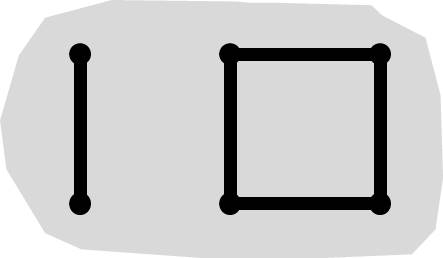}  &0        &2    \\
          &\includegraphics[width=0.8in]{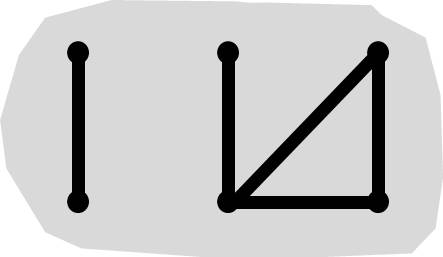}  &-1       &1    \\
 14.    &\includegraphics[width=0.8in]{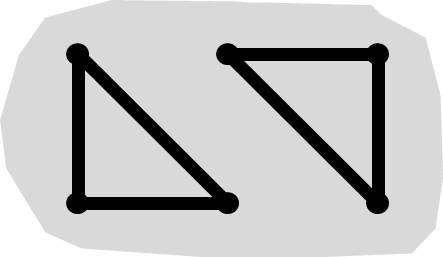}  &4         &0    \\
 \hline   
\end{tabular}
\end{center}

In short, adding two quantities, $c_6$ and  $\binom{|E(G)|}{3}$ would allow us to eliminate, as it can be seen from the tables, most of the graphs and leave only twelve of them. They are given in Figure \ref{fig6} with the same numeration as in tables. Denote $n_i$ - number of graphs isomorphic to the graph enumerated by  $i, 1\leq i \leq12$, from the Figure \ref{fig6}. Then,

\begin{align}
&c_6+\binom{|E(G)|}{3}=4n_1-2n_2+2n_3+2n_4+4n_5+2n_6+2n_7-2n_8+2n_9+\nonumber \\
&2n_{10}+2n_{11}-2n_{12}+2n_{13}+4n_{14}+e_4+e_5.
\label{eq2}
\end{align}

\begin{figure}
	\includegraphics[width=0.8\textwidth]{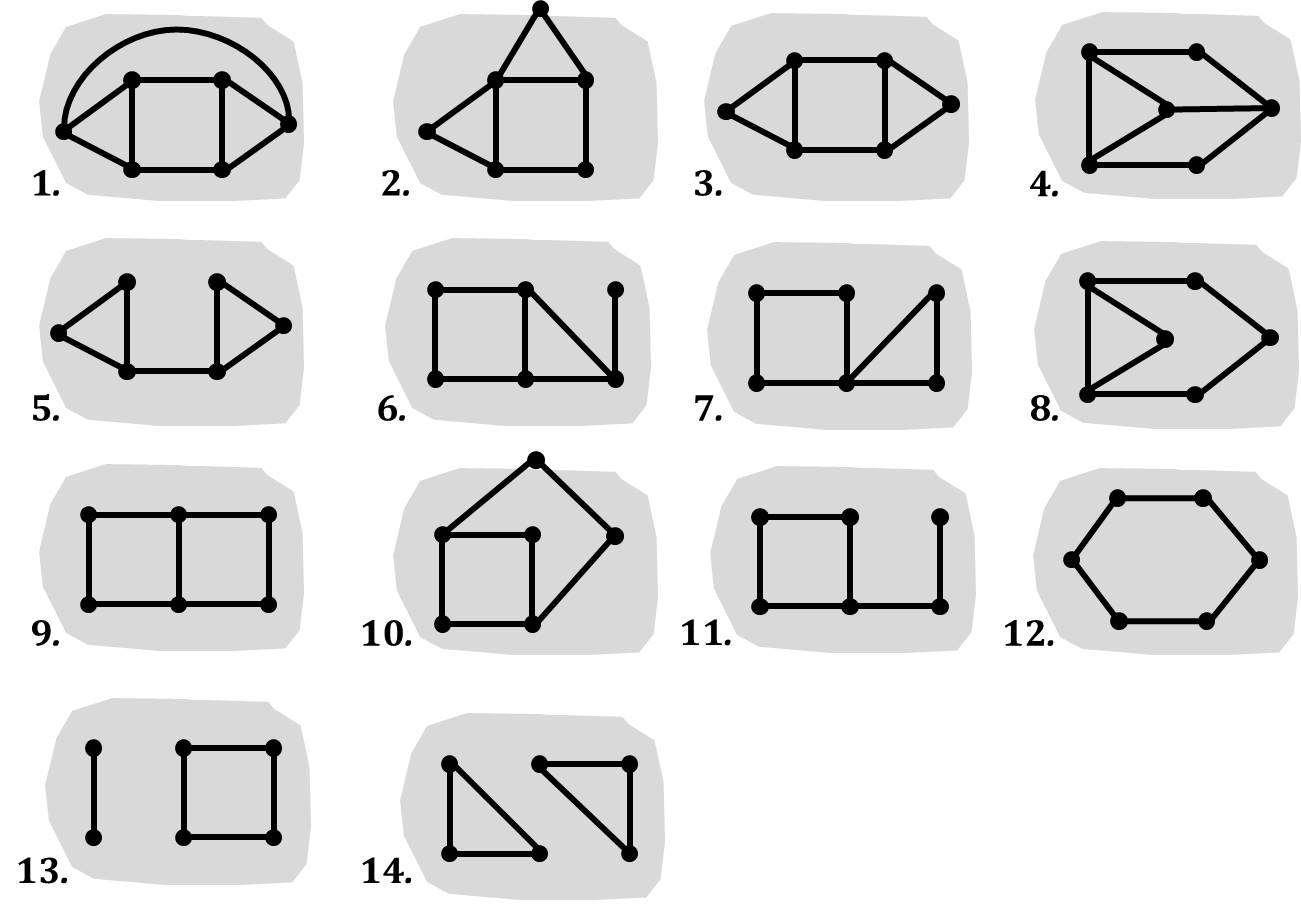}
		\centering
		\caption{The only induced subgraphs that have not been eliminated by summation $c_6+\binom{|E(G)|}{3}$. The enumeration is the same as in Table 1 and 2.}
		\label{fig6}
\end{figure}

To proceed further, we need the next rather lengthy proposition. It will allow us to tie up all the quantities on the right hand side of (\ref{eq2}).

\begin{prop}
The following equalities hold:
\begin{align}
n_2&=\frac{1}{2}nk(k-2);\\
n_4+n_8&=nk(k-2)(k-4);\\
6n_1+n_4&=\frac{1}{2}nk(k-2);\\
3n_1+n_3&=\frac{1}{4}nk(k-2);\\
3n_1+n_4+n_9&=\frac{1}{4}nk(k-2)(k-3);\\
n_1+n_3+n_5+n_{14}&=\frac{1}{12}nk(\frac{nk}{6}-1)-\frac{1}{8}nk(k-2);\\
3n_1+2n_4+n_6+n_7+2n_9+n_{10}+n_{11}+n_{13}&=\frac{1}{8}nk(k-2)(\frac{nk}{2}-4k+4).
\end{align}
\label{relations}
\end{prop}

\begin{proof}
Take a quadrilateral from $G$. Complete on its two adjacent sides triangles. We will get the unique graph of type 2 (Figure \ref{fig6}). So using Proposition 3: \[n_2=4p_4=\frac{1}{2}nk(k-2).\]

Take a pentagon from $G$. Complete a triangle on one of its sides. We can get the graph of type 4 or type 8 and no other one. Thus, \[n_4+n_8=5 p_5=nk(k-2)(k-4).\]

Take a triangle in $G$. Choose a vertex adjacent to one of the three vertices of the triangle, out of $3(k-2)$ possible ones. Complete it uniquely the way we did it in Figure 5a. We have two possible configurations: 4 and 1. If we get graph 1, it should be counted six times as it can be obtained in six different ways. \[6n_1+n_4=p_3\cdot 3(k-2)=\frac{1}{2}nk(k-2).\]

Take again a quadrilateral from $G$. Similarly, complete two triangles but this time on its opposite sides. We can get either graph 3 or 1. If we get graph 1, we have to count it three times as it has three distinct quadrilaterals, we could start with. \[3n_1+n_3=2p_4=\frac{1}{4}nk(k-2).\]

All three graphs:1, 4 and 9, consist of two quadrilaterals sharing an edge. Any edge belongs to exactly $k-2$ quadrilaterals, from which we can choose pairs of quadrilaterals. Graph 1 again is counted three times as it has three pairs of such quadrilaterals. So,\[ 3n_1+n_4+n_9=|E(G)| \binom{k-2}{2}=\frac{1}{4}nk(k-2)(k-3).\]

The four graphs: 1, 3, 5, and 14 are all consist of exactly two triangles that do not share a vertex. In order to find all such configurations we simply need to subtract from all possible pairs of triangles those that DO share a vertex. Remember, they can share at most one vertex due to condition 1. Thus, \[n_1+n_3+n_5+n_{14}=\binom{p_3}{2}-n\binom{k/2}{2}= \frac{1}{12}nk(\frac{nk}{6}-1)-\frac{1}{8}nk(k-2).\]

Finally, the last relation bonds the graphs that can be obtained by choosing a quadrilateral and an edge that is not incident to any of the vertices of the quadrilateral. Notice, once a quadrilateral is chosen, the choice of an edge uniquely defines the six-vertex graph. Thus, the coefficients in front of the quantities depend only on number of quadrilaterals the graph has.
\begin{align*}
3n_1+2n_4+n_6+n_7+2n_9+n_{10}+n_{11}+n_{13}&=p_4(|E(G)|-4(k-2)-4)\\ &=\frac{1}{8}nk(k-2)(\frac{nk}{2}-4k+4). \end{align*}

\end{proof}

Now, when we are equipped with all the relations from Proposition \ref{relations}, we can proceed with (2). But first, let us remind ourselves what we are trying to achieve with all these cumbersome calculations. We want to find $n_{12}$ - the number of subgraphs of type 12 in $G$, namely hexagons (Figure \ref{fig6}).

Rewrite (2),
\begin{align*}
&2n_1-n_2+n_3+n_4+2n_5+n_6+n_7-n_8+n_9+n_{10}+n_{11}- n_{12}+n_{13}+2n_{14}\\
&=\frac{1}{2}(c_6+\binom{|E(G)|}{3}-e_4-e_5).
\end{align*}
Subtracting (9) from this expression, we obtain,
\begin{align*}
&-n_1-n_2+n_3 -n_4+2n_5-n_8-n_9- n_{12}+2n_{14}\\
&=\frac{1}{2}(c_6+\binom{|E(G)|}{3}-e_4-e_5)-\frac{1}{8}nk(k-2)(\frac{nk}{2}-4k+4).
\end{align*}
Further subtracting double of (8), we get rid of $n_5$ and $n_{14}$,
\begin{align*}
&-3n_1-n_2-n_3 -n_4-n_8-n_9- n_{12}=\frac{1}{2}(c_6+\binom{|E(G)|}{3}-e_4-e_5) \\
&-\frac{1}{8}nk(k-2)(\frac{nk}{2}-4k+4)-\frac{1}{6}nk(\frac{nk}{6}-1)+\frac{1}{4}nk(k-2).
\end{align*}
The right hand side of the expression is getting horrible and Wolfram Alfa \cite{Wolfram} here is of no help (or we just didn't find the way to use it properly), but we should not worry about it at the moment and rather concentrate on the left-hand side solely. Using (3), (4) and (6), we get
\begin{align*}
&-n_9- n_{12}=\frac{1}{2}(c_6+\binom{|E(G)|}{3}-e_4-e_5) -\frac{1}{8}nk(k-2)(\frac{nk}{2}-4k+4) \\
&-\frac{1}{6}nk(\frac{nk}{6}-1)+\frac{1}{4}nk(k-2)+\frac{3}{4}nk(k-2) +nk(k-2)(k-4).
\end{align*}

Next, we express $n_9$ through needed $n_4$, using (7) and (5).
\[n_9=\frac{1}{4}nk(k-2)(k-3)-\frac{1}{4}nk(k-2)-\frac{n_4}{2}.\]

Substituting,
\begin{align*}
&\frac{n_4}{2}- n_{12}=\frac{1}{2}(c_6+\binom{|E(G)|}{3}-e_4-e_5) -\frac{1}{8}nk(k-2)(\frac{nk}{2}-4k+4) \\
&-\frac{1}{6}nk(\frac{nk}{6}-1)+\frac{3}{4}nk(k-2) +nk(k-2)(k-4) + \frac{1}{4}nk(k-2)(k-3).
\end{align*}
Denoting right hand side by $-F(n,k)$, we have: \[\frac{n_4}{2}- n_{12}=-F(n,k).\]
Notice also that from (5) and (6) $n_4=2n_3$. Thus, \[n_{12}=F(n,k)+n_3.\]
Tedious calculations are required in order to proceed further with the expression on the right hand side. The challenge is the coefficient $c_6$ that is inside $F(n,k)$. It can be easily calculated numerically for a particular value of $n$ and $k$, using the relation  \[c_6=k\sum_{i=0}^{5} \binom{r_1}{5-i}\binom{r_2}{i}\lambda_1^{5-i}\lambda_2^{i}+\sum_{i=0}^{6} \binom{r_1}{6-i}\binom{r_2}{i}\lambda_1^{6-i}\lambda_2^{i}.\]
Here $\lambda_1, \lambda_2$ = eigenvalues of an adjacency matrix $A(G)$, and $r_1,r_2$ their respective multiplicities. In particular, the characteristic polynomial $P_G(x)=(x-k)(x-\lambda_1)^{r_1}(x-\lambda_2)^{r_2}.$

\begin{center}
\noindent
\begin{tabular}{ |r|r|r|r|r|}
  \hline
  \multicolumn{3}{|l|}{Table 3: The values of $c_6$ for several orders of $G$} \\
  \hline
 n           &k    & $c_6$                    \\ 
 \hline
 9           &4    &-168                       \\
 99         &14  &-47,288,703            \\
 243       &22  &-2,975,686,065       \\
 6,273    &112 &-7,204,770,339,625,320 \\
 494,019&994 &-2,466,795,174,682,153,663,896,408 \\
 \hline   
\end{tabular}
\end{center}
The numerical values of $c_6$ for several orders of $G$ are given in the above table. The calculations are done using Julia programming language \cite{Julia}. Further heavily relying on computation machinery of Wolfram Alpha \cite{Wolfram} with the use of the additional relations for eigenvalues and their multiplicities (see West, 2-nd ed., p.466 \cite{West}) such as: 
\begin{align*}
\lambda_1+\lambda_2&=-1;\\
\lambda_1\lambda_2&=-(k-2);\\
r_1+r_2&=n-1;
\end{align*}
we will obtain a formula \[c_6=-\frac{1}{576}nk(k-2)(3k^5+6k^4-84k^3+116k^2+124k-240).\]
Finally, plugging everything back into $F(n,k)$, the expression for $n_{12}$ simplifies to
\[n_{12}=\frac{1}{12}nk(k-2)(2k^2-21k+53)+n_3.\]
By this and $n_3 \geq 0$, we have proven the following statement:

\begin{theorem}
The number of hexagons in $G$ is at least $\frac{1}{12}nk(k-2)(2k^2-21k+53)$.
\end{theorem}

\section{Conclusion}

In this paper we have studied the structure of a class of strongly regular graphs with parameters $\lambda=1$ and $\mu=2$. We have shown that the lower bound for the number of hexagons in such graphs is $\frac{1}{12}nk(k-2)(2k^2-21k+53)$. We conjecture that the lower bound is indeed the true value for $p_6$ due to many symmetries broken otherwise. This bound is achieved when $n_3=0$, which in turn meaning that two triangles in $G$ connected through two edges are necessarily connected through the third one. Given such condition, Makhnev \cite{Makhnev} has proved that $srg(99,14,1,2)$ doesn't exist.

\begin{conjecture} The number of hexagons in strongly regular graphs with parameters $\lambda=1$ and $\mu=2$ is equal to $\frac{1}{12}nk(k-2)(2k^2-21k+53)$.
\end{conjecture}

Several things worth noticing. First, given the conjecture is true, Makhnev's condition holds not only for a graph with $n=99$ and $k=14$ but for the entire family of strongly regular graphs with $\lambda=1$ and $\mu=2$. This can be observed for the case when $k=4$ or Paley 9. Some preliminary checks show that it holds for another known graph from the family - the Berlekamp–Van Lint–Seidel graph, for $k=22$ \cite{Berlekamp}. Second, all the graphs, except of trivial case $K_3$, must be built of Paley 9 graphs as building blocks if the conjecture is true. In particular, both of the yet unknown graphs for $k=112$ and $k=994$ in that case have more coarse structure. Their $P_9$-built structure might give us an insight on their existence as well.


\end{document}